\documentclass[preprint,12pt]{elsarticle}
\usepackage[T1]{fontenc}
\usepackage{amssymb,amsmath,graphicx,color,amsthm}

\newtheorem{theorem}{Theorem}

\newtheorem{conjecture}{Conjecture}

\renewcommand{\int}{{\rm int}}

\begin{document}

\begin{frontmatter}

\title{Improved bounds for some facially constrained colorings}

\author[fis,famnit]{Kenny \v{S}torgel\fnref{arrs}}
\ead{kennystorgel.research@gmail.com}

\fntext[mr]{This research was funded by a Young Researchers Grant from the Slovenian Research Agency.}

\address[fis]{Faculty of Information Studies, Novo mesto, Slovenia}
\address[famnit]{University of Primorska, FAMNIT, Koper, Slovenia}

\begin{abstract}
	A facial-parity edge-coloring of a $2$-edge-connected plane graph is a facially-proper edge-coloring in which every face is incident with zero or an odd number of edges of each color. A facial-parity vertex-coloring of a $2$-connected plane graph is a facially-proper vertex-coloring in which every face is incident with zero or an odd number of vertices of each color. 
	Czap and Jendro\v{l} (in \textit{Facially-constrained colorings of plane graphs: A survey, Discrete Math. 340 (2017), 2691--2703}), 
	conjectured that $10$ colors suffice in both colorings. We present an infinite family of counterexamples to both conjectures.
	
	A facial $(P_{k}, P_{\ell})$-WORM coloring of a plane graph $G$ is a coloring of the vertices 
	such that $G$ contains no rainbow facial $k$-path and no monochromatic facial $\ell$-path. 
	Czap, Jendro\v{l} and Valiska (in \textit{WORM colorings of planar graphs, Discuss. Math. Graph Theory 37 (2017), 353--368}), 
	proved that for any integer $n\ge 12$ there exists a connected plane graph on $n$ vertices, with maximum degree at least $6$, 
	having no facial $(P_{3},P_{3})$-WORM coloring. 
	They also asked if there exists a graph with maximum degree $4$ having the same property. 
	We prove that for any integer $n\ge 18$, there exists a connected plane graph, with maximum degree $4$, with no facial $(P_{3},P_{3})$-WORM coloring. 
\end{abstract}


\end{frontmatter}

\section{Introduction}



Historically, the Four Color Problem became a great motivation for researchers to study many different types of colorings restricted to planar graphs. Among some of the recent studies are those that study colorings of plane graphs where certain constraints are given by the faces. Czap and Jendro\v{l}~\cite{survey} wrote a survey devoted to presenting many of such colorings, in which they also presented a number of open problems. In this note, we consider three types of facially constrained colorings of plane graphs.\\
\indent All graphs considered are finite and planar with a fixed embedding, where $V$ denotes the set of vertices, $E$ denotes the set of edges and $F$ denotes the set of faces of $G$. Let $C$ be a set of colors. For simplicity we take $C = \mathbb{N}$, where each positive integer denotes a single color. A \textit{vertex-coloring} of a graph $G$ is a function $f: V(G)\mapsto \mathbb{N}$, assigning to each vertex of $G$ a color $c\in \mathbb{N}$. An \textit{edge-coloring} of a graph $G$ is a function $g: E(G)\mapsto \mathbb{N}$, assigning to each edge of $G$ a color $c\in \mathbb{N}$. We say that a vertex-coloring (edge-coloring) is \textit{proper} if every two adjacent vertices (edges) receive distinct colors. An edge-coloring is \textit{facially-proper} if for every face $\alpha\in F$, every pair of incident edges appearing consecutively on the boundary of the face $\alpha$ receives distinct colors.\\


\section{Facial-parity edge-coloring}

A \textit{facial-parity edge-coloring} of a $2$-edge-connected plane graph $G$, is a facially-proper coloring of the edges of $G$, where every face $\alpha$ of $G$ is incident with either odd or zero edges of each color. We denote with $\chi_{fp}'(G)$ the minimum number $k$, for which there exists a facial-parity edge-coloring of $G$ with $k$ colors. Facial-parity edge-coloring was first studied by Czap, Jendro\v{l} and Kardo\v{s} in~\cite{facialparityedge}, where they proved that $92$ colors suffice to color every $2$-edge-connected plane graph. This bound was later improved by Czap et al.~\cite{paritypseudo} to $20$ colors. The best known upper bound so far is $16$ colors, due to Lu\v{z}ar and \v{S}krekovski~\cite{parity16}. In~\cite{parityouterplane}, an example of an outerplane graph is presented, namely two cycles $C_{5}$ sharing a single vertex, which needs $10$ colors. Later Czap and Jendro\v{l}~\citep{survey} proposed the following conjecture.

\begin{conjecture}
	If $G$ is a $2$-edge-connected plane graph, then $\chi_{fp}'(G)\le 10$.
\end{conjecture}

The \textit{Theta graph}, $\Theta_{i,j,k}$, is the graph consisting of two distinct vertices joined by three internally vertex disjoint paths of lengths $i$, $j$, and $k$. 
We present the following result, which rejects the conjecture and shows that the lower bound is at least $12$.
\begin{figure}[h!]
\center
\includegraphics[scale=0.8]{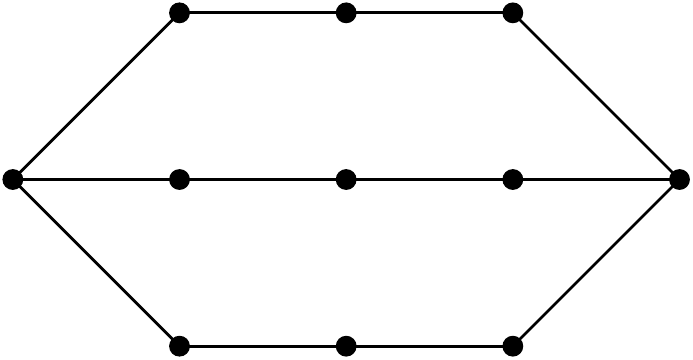}
\caption{The graph $\Theta_{4,4,4}$ with $12$ edges and $\chi_{fp}'(G) = 12$.}\label{ThetaGraph}
\end{figure}

\begin{theorem}\label{FPedge}
For any integer $k\ge 3$, there exists a $2$-edge-connected plane graph $G$ with $4k$ edges and $\chi_{fp}'(G) = 12$.
\end{theorem}

\begin{proof}
Let $G$ be any Theta graph. Fix some plane embedding of $G$ (e.g.~see Fig~\ref{ThetaGraph}). 
Clearly $G$ is $2$-edge-connected. Now note that $G$ can be edge decomposed into three internally vertex-disjoint paths $P_{1}$, $P_{2}$ and $P_{3}$, where $P_{i}$ and $P_{j}$, $1\le i<j\le 3$, are both incident to the unique face $\alpha_{ij}$. Let $f:E(G)\rightarrow \mathbb{N}$ be any facial-parity edge-coloring of $G$. First suppose that some color $c$ appears an even number of times on the edges of some path $P_{i}$. Without loss of generality, we can assume that $i=1$. Since $P_{1}$ is incident to both $\alpha_{12}$ and $\alpha_{13}$, it follows that the color $c$ must appear an odd number of times on the edges of both $P_{2}$ and $P_{3}$, but then it appears an even number of times on the edges of the face $\alpha_{23}$, a contradiction. It follows directly that no color can appear on two distinct paths $P_{i}$ and $P_{j}$ at the same time. Therefore, the number of colors needed to color the edges of $G$ is the sum of the number of colors needed to color the edges of each $P_{i}$ individually. Let us consider again a single path $P\in \{P_{1}, P_{2}, P_{3}\}$ and let the length of $P$ be $\ell$. We consider the following four cases:\\
\indent \textit{Case 1:} If $\ell = 2m$ for some $m\in \mathbb{N}$, where $m$ is odd, then we can properly color the edges of $P$ with exactly two colors $c_{1}$ and $c_{2}$, each appearing $m$ times on $P$.\\
\indent \textit{Case 2:} If $\ell = 2m+1$ for some $m\in \mathbb{N}$, where $m$ is even, then we can color the edges of $P$ with exactly three colors $c_{1}$, $c_{2}$ and $c_{3}$, where each of the colors $c_{1}$ and $c_{2}$ appears $m-1$ times on $P$ and the color $c_{3}$ appears $3$ times on $P$ (the exception is when $m=2$, in which case we use color $c_{1}$ $3$ times and each of the colors $c_{2}$ and $c_{3}$ only once).\\
\indent \textit{Case 3:} If $\ell = 2m+1$ for some $m\in \mathbb{N}$, where $m$ is odd, then we can color the edges of $P$ with exactly three colors $c_{1}$, $c_{2}$ and $c_{3}$, where each of the colors $c_{1}$ and $c_{2}$ appears $m$ times on $P$ and the color $c_{3}$ appears only once on $P$.\\
\indent \textit{Case 4:} If $\ell = 4m$ for some $m\in \mathbb{N}$, then we can color the edges of $P$ with exactly four colors $c_{1}$, $c_{2}$, $c_{3}$ and $c_{4}$, where each of the colors $c_{1}$ and $c_{2}$ appears $2m-1$ times on $P$ and each of the colors $c_{3}$ and $c_{4}$ appears only once on $P$.\\
It follows that if each of the paths $P_{i}$ has length divisible by $4$, then $\chi_{fp}'(G) = 12$, thus proving the theorem. The smallest such case is depicted in Fig.~\ref{ThetaGraph}, where all three paths are of length $4$ and $G$ has $12$ edges.
\end{proof}

\section{Facial-parity vertex-coloring}

A \textit{facial-parity vertex-coloring} of a plane graph $G$ is a facially-proper coloring of the vertices of $G$ in such a way, that every face $\alpha$ of $G$ is incident with either odd or zero vertices of each color. We denote by $\chi_{fp}(G)$ the minimum number $k$, for which there exists a facial-parity vertex-coloring of $G$ with $k$ colors. Czap, Jendro\v{l} and Voigt~\cite{vertexparity118} proved that $118$ colors are sufficient to color every $2$-connected plane graph $G$. Kaiser et al.~\cite{vertexparity97} improved the bound to $97$ colors, which is the best known bound so far. Czap~\cite{parityvertex10example} showed that there exists an outerplane graph which needs $10$ colors. In~\cite{wang2examplesfor10} the authors proved that there exist only two $2$-connected outerplane graphs with $\chi_{fp} = 10$. Motivated by that, Czap and Jendro\v{l}~\cite{survey} proposed the following.

\begin{conjecture}
Every $2$-connected plane graph admits a facial-parity vertex-coloring with at most $10$ colors.
\end{conjecture}

With the following theorem, we prove that there exists an infinite family of $2$-connected plane graphs with $\chi_{fp}(G) = 12$.

\begin{figure}[h!]
\center
\includegraphics[scale=0.8]{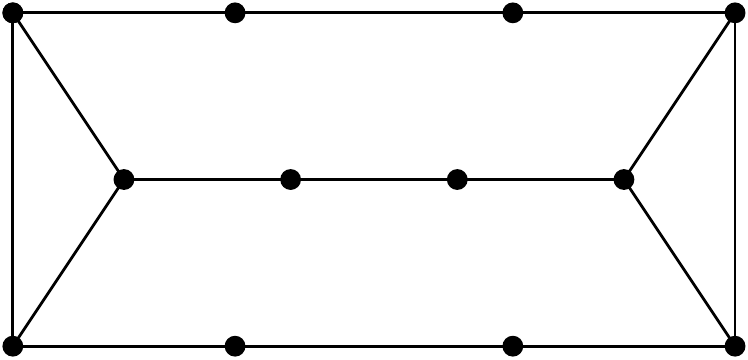}
\caption{The line graph of the graph $\Theta_{4,4,4}$ with $12$ vertices and $\chi_{fp}(G) = 12$.}\label{LThetaGraph}
\end{figure}

\begin{theorem}\label{FPvertex}
For any integer $k\ge 3$, there exists a $2$-edge-connected plane graph $G$ with $4k$ vertices and $\chi_{fp}(G) = 12$.
\end{theorem}

\begin{proof}
First observe that the line graphs of the graphs described in the proof of Theorem~\ref{FPedge} are $2$-edge-connected and planar. 
Let $H$ be any Theta graph and $G$ be its line graph (Fig.~\ref{LThetaGraph} represents a particular embedding of a line graph of a graph from Fig.~\ref{ThetaGraph}). It is clear that the number of vertices of $G$ is at least $12$ and divisible by $4$. Observe that any facial-parity vertex coloring of $G$ defines a facial-parity edge-coloring of $H$ and any facial-parity edge-coloring of $H$ defines a facial-parity vertex coloring of $G$. It follows from Theorem~\ref{FPedge} that $\chi_{fp}(G) = 12$.
\end{proof}

\section{Facial $(P_{3},P_{3})$-WORM coloring}

We say that a vertex-coloring of a graph is \textit{rainbow} (e.g. see~\cite{bujtas2,bujtas}), if every pair of its vertices receives distinct colors. On the other hand, we say that a vertex-coloring of a graph is \textit{monochromatic} (e.g. see~\cite{tuza}), if every vertex receives the same color. Given three graphs $G, H$ and $F$, an $(H,F)$-\textit{WORM} coloring of $G$ is a coloring of the vertices of $G$ in such a way, that no subgraph isomorphic to $H$ is rainbow and no subgraph isomorphic to $F$ is monochromatic. The study of WORM colorings was initiated by Voloshin~\cite{voloshin} and they have been extensively studied since (e.g. see~\cite{fworm,survey,worm}). In~\cite{fworm}, the authors studied a special case of $(H,F)$-WORM coloring, namely an $F$-WORM coloring, where $H$ and $F$ are isomorphic. The idea of an $F$-WORM coloring was first introduced by Goddard, Wash, and Xu~\cite{WORMstart,WORMstart2}.\\
\indent Facially constrained WORM colorings were studied in several papers and in~\cite{worm} the authors introduce a \textit{facial} $(P_{k},P_{\ell})$-\textit{WORM coloring}. That is a vertex-coloring of a plane graph $G$, having no rainbow facial $P_{k}$ and no monochromatic facial $P_{\ell}$. Furthermore, they prove that the graph $G$, with $\Delta(G) = 6$, where $\Delta$ denotes the maximum degree of a graph, obtained from the graph in Fig.~\ref{WormGraph} by contracting the edges of the $4$-cycles, incident with the outer face, has no facial $(P_{3},P_{3})$-WORM coloring. They also asked a question whether there exist plane graphs, with maximum degree $4$, having no facial $(P_{3},P_{3})$-WORM coloring. We answer the question in affirmative.

\begin{figure}[h!]
\center
\includegraphics[scale=0.8]{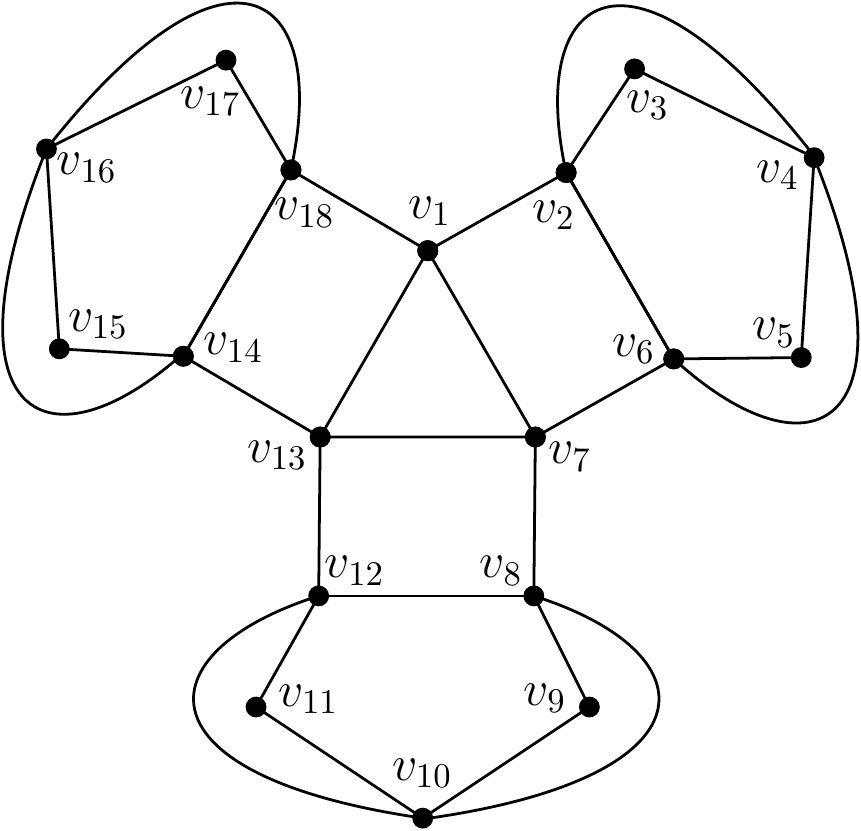}
\caption{A graph with $18$ vertices, with maximum degree $4$, having no facial $(P_{3},P_{3})$-WORM coloring.}\label{WormGraph}
\end{figure}

\begin{theorem}\label{WORM}
For any integer $n\ge 18$, there exists a connected plane graph $G$ on $n$ vertices with $\Delta(G) = 4$, having no facial $(P_{3},P_{3})$-WORM coloring. 
\end{theorem}

\begin{proof}
Let $G$ be a planar graph on $18$ vertices with its planar embedding as given in Fig.~\ref{WormGraph}. Suppose that $G$ admits a facial $(P_{3},P_{3})$-WORM coloring. Note that the vertices $v_{1}$, $v_{7}$ and $v_{13}$ form a face of size $3$. Since there is no rainbow facial path $P_{3}$ and no monochromatic facial path $P_{3}$ in $G$, it follows that exactly two of the three vertices share a color. Without loss of generality, we can assume that $v_{1}$ and $v_{7}$ are colored with the same color $c_{1}$. Now observe that the vertices $v_{1}$, $v_{2}$, $v_{6}$ and $v_{7}$ form a face of size $4$. Since the vertices $v_{1}$ and $v_{7}$ are adjacent and both colored with color $c_{1}$, and there is no monochromatic facial $P_{3}$, it follows that the vertices $v_{2}$ and $v_{6}$ are both colored with a color different from $c_{1}$. We also know that both $v_{2}$ and $v_{6}$ are colored with the same color $c_{2}$, since there is no rainbow facial $P_{3}$ in $G$. Now observe that none of the vertices $v_{3}$, $v_{4}$ or $v_{5}$ is colored with color $c_{2}$, otherwise we would obtain a monochromatic facial $P_{3}$. Let $c_{3}$ be the color of the vertex $v_{3}$. Suppose that the color of the vertex $v_{4}$ is different from $c_{3}$. Then we obtain a rainbow facial $P_{3}$ on the face of size $3$, formed by vertices $v_{2}$, $v_{3}$ and $v_{4}$, a contradiction. Thus, $v_{4}$ must be colored with the color $c_{3}$. Consider now the vertex $v_{5}$. If $v_{5}$ is colored with the color $c_{3}$, then we obtain a monochromatic facial $P_{3}$, formed by the vertices $v_{3}$, $v_{4}$ and $v_{5}$. It follows that $v_{5}$ must receive a color different from $c_{2}$ and $c_{3}$, say $c_{4}$, but then we obtain a rainbow facial $P_{3}$, formed by the vertices $v_{4}$, $v_{5}$ and $v_{6}$, a contradiction.\\
\indent If $n>18$, then take the graph $G$ and any connected planar graph $H$, with $\Delta(H) \le 4$, having a planar embedding such that the outer face contains a vertex $v$ of degree at most $3$. Then place $H$ in any face $\alpha$ of size $3$, distinct from the face formed by the vertices $v_{1}$, $v_{7}$ and $v_{12}$, and add the edge from $v$ to the only vertex of the face $\alpha$, with degree less than $4$.
\end{proof}

\section{Conclusion}

Examples presented in the previous sections provide new bounds for their corresponding coloring problems. New results prove that we need at least $12$ colors for facial-parity edge-coloring of any $2$-edge-connected planar graph. The current upper bound however remains $16$, and hence there is still a gap of $4$ colors between the two bounds. On the other hand, for the vertex version, even though we presented an example which increases the lower bound by two, there is still a huge gap between this new lower bound of $12$ and currently best known upper bound of $97$.\\
\indent In regards with facial WORM vertex-coloring of plane graphs, it is known that not all plane graphs have a $(P_{3},P_{3})$-WORM coloring. Czap, Jendro\v{l} and Valiska~\citep{worm} presented some results about a $(P_{3},P_{4})$-WORM coloring, yet their conjecture that every connected plane graph admits a $(P_{3},P_{4})$-WORM coloring with $2$ colors remains open.

\bibliographystyle{plain}
{\small
	\bibliography{MainBase}
}

\end{document}